\documentclass[12pt]{amsproc}

\usepackage{amsmath, amssymb}

\newtheorem{theorem}{\sc Theorem}[section]
\newtheorem{lemma}[theorem]{\sc Lemma}
\newtheorem{proposition}[theorem]{\sc Proposition}
\newtheorem{corollary}[theorem]{\sc Corollary}

\newtheorem{hypothesis}[theorem]{\sc Hypothesis}

\newcommand{\F}{\mathcal{F}}

\newcommand{\CF}{\mathcal{CF}}

\newcommand{\C}{\mathcal{C}}
\newcommand{\X}{\mathcal{X}}

\newcommand{\g}{{t}} % here we can change the letter denoting the coset

\begin{document}

\title[Coverings by cosets]
{On profinite groups with commutators covered by countably many cosets %On profinite groups with word values  covered by countably many cosets
% \thanks{This research was partially supported by Universit\`a di Padova (Progetto di Ricerca di Ateneo:``Invariable generation of groups"). The third author was also supported by FAPDF and CNPq.}
}
%\author{Eloisa Detomi \and Marta Morigi \and Pavel Shumyatsky}
%\institute{E. Detomi \at Dipartimento di Matematica, Universit\`a di Padova, Via Trieste 63,  35121 Padova, Italy  \\ \email{detomi@math.unipd.it}
% \and  M. Morigi \at Dipartimento di Matematica, Universit\`a di Bologna,  Piazza di Porta San Donato 5,  40126 Bologna, Italy \\ \email{marta.morigi@unibo.it}
%\and P. Shumyatsky \at  Department of Mathematics, University of Brasilia,  Brasilia-DF,  70910-900 Brazil \\ \email{pavel@unb.br}  }
\author[E. Detomi]{Eloisa Detomi}
\address{Dipartimento di Matematica, Universit\`a di Padova \\
 Via Trieste 63\\ 35121 Padova \\ Italy}
\email{detomi@math.unipd.it}
\author[M. Morigi]{Marta Morigi}
\address{Dipartimento di Matematica, Universit\`a di Bologna\\
Piazza di Porta San Donato 5 \\ 40126 Bologna\\ Italy}
\email{marta.morigi@unibo.it}
\author[P. Shumyatsky]{Pavel Shumyatsky}
\address{Department of Mathematics, University of Brasilia\\
Brasilia-DF \\ 70910-900 Brazil}
\email{pavel@unb.br}

%\thanks{This research was partially supported by Universit\`a 
%di Padova (Progetto di Ricerca di Ateneo:``Invariable generation of groups"). The second author was also supported by GNSAGA (INDAM), and the third author by  FAPDF and CNPq.}


\begin{abstract} Let $w$ be a group-word. Suppose that the set of all $w$-values in a profinite group $G$ is contained in a union of countably many cosets of subgroups. We are concerned with the question to what extent the structure of the verbal subgroup $w(G)$ depends on the properties of the subgroups. We prove the following theorem.

Let $\C$ be a class of groups closed under taking subgroups, quotients, and such that in any group the product of finitely many normal $\C$-subgroups is again a $\C$-subgroup. If $w$ is a multilinear commutator and $G$ is a profinite group such that the set of all $w$-values is contained in a union of countably many cosets $g_iG_i$, where each $G_i$ is in $\C$, then the verbal subgroup $w(G)$ is virtually-$\C$.

This strengthens several known results.\end{abstract}

%%% ----------------------------------------------------------------------
\maketitle
%%% ----------------------------------------------------------------------

\section{Introduction} A  covering of a group $G$ is a family $\{S_i\}_{i\in I}$ of subsets of $G$ such that $G=\bigcup_{i\in I}\,S_i$. The famous result of B.H. Neumann states that if $\{S_i\}$ is a finite covering of $G$ by cosets of subgroups, then $G$ is actually covered by the cosets $S_i$ corresponding to subgroups of finite index in $G$ \cite{neumann}. Therefore whenever a group $G$ is covered by finitely many cosets of subgroups it is natural to expect that some structural information about $G$ can be deduced from the properties of the subgroups. In other words, the general question is to what extent properties of the covering subgroups impact the structure of $G$.

In recent years some ``verbal" variations of these questions became a subject of research activity. Given a group-word $w=w(x_1,\dots,x_n)$, we think of it primarily as a function of $n$ variables defined on any given group $G$. We denote by $w(G)$ the verbal subgroup of $G$ generated by the values of $w$. When the set of all $w$-values in a group $G$ is contained in a union of finitely many subgroups (or cosets of subgroups) we wish to know   whether the properties of the covering subgroups have impact on the structure of the verbal subgroup $w(G)$. The present article deals with the situation when $G$ is a profinite group.

In the context of profinite groups all the usual concepts of group theory are interpreted topologically. In particular, by a subgroup of a profinite group we mean a closed subgroup. A subgroup is said to be generated by a set $S$ if it is topologically generated by $S$. Thus, the verbal subgroup $w(G)$ in a profinite group $G$ is a minimal closed subgroup containing the set of $w$-values. One important tool for dealing with  the ``covering" problems in profinite groups is the classical Baire's category theorem (cf \cite[p.\ 200]{Ke}): If a locally compact Hausdorff space is a union of countably many closed subsets, then at least one of the subsets has non-empty interior. It follows that if a profinite group is covered by countably many cosets of subgroups, then at least one of the subgroups is open. Thus, in the case of profinite groups we can successfully deal with problems on countable coverings rather than just finite ones.

The reader can consult the articles \cite{surveyrendiconti,as,DMS1,DMS-revised,DMS-nilpotent, Snilp} 
 for results on countable coverings of word-values by 
subgroups. One of the results obtained in \cite{DMS-nilpotent} is that if $w$ is a multilinear commutator and $G$ is a profinite group, 
then $w(G)$ is finite-by-nilpotent if and only if the set of $w$-values in $G$ is covered by countably many finite-by-nilpotent subgroups 
(see Section \ref{sec:prelim} for the definition of multilinear commutator). It is easy to see that the above result is no longer true 
if the set of $w$-values in $G$ is covered by countably many cosets of finite-by-nilpotent subgroups. This can be exemplified by any 
profinite group $G$ having $w(G)$ virtually nilpotent but not finite-by-nilpotent. In the present article we study groups in which 
the set of $w$-values is covered by countably many cosets of $\mathcal{C}$-subgroups, where $\C$ is a class of groups closed under 
taking subgroups, quotients, and such that in any group the product of finitely many normal $\C$-subgroups is again a $\C$-subgroup.

Our main result is as follows. 

\begin{theorem}\label{main} Let $\C$ be a class of groups closed under taking subgroups, quotients, and such that in any group the product of finitely 
many normal 
$\C$-subgroups is again a $\C$-subgroup.  Let $w$ be a multilinear commutator word. The verbal subgroup $w(G)$ of a profinite group $G$ is virtually-$\mathcal{C}$ if and only if the set of $w$-values in $G$ is covered by countably many cosets of $\mathcal{C}$-subgroups. 
\end{theorem}

We note that many natural classes of groups have the properties as the class $\C$ in the above theorem. For instance, $\C$ can be the class of nilpotent, pronilpotent, locally nilpotent, or soluble groups. Further examples include torsion groups and groups of finite rank. 
It is been known for sometime that if $w$ is a multilinear commutator  and
a profinite group $G$ has countably many soluble subgroups whose union contains all $w$-values, 
 then $w(G)$ is virtually soluble \cite[Theorem 7]{adm2012}. If $G$ has countably many torsion subgroups 
 (or subgroups of finite rank) whose union contains all $w$-values, then $w(G)$ is torsion (or of finite rank) \cite{DMS1}. 
 Obviously, Theorem \ref{main} extends these results. Moreover, in the case where $\C$ is the class of all finite groups, we obtain that 
 the set of $w$-values in a profinite group $G$ is countable if and only if $w(G)$ is finite. This was one of the main results in \cite[Theorem 1.1]{jpaa2}. 

A few words about the tools employed in the proof of Theorem \ref{main}. Rather specific combinatorial techniques for handling multilinear 
commutator words were developed in \cite{fernandez-morigi, DMS1,DMS-revised}.  The present article is based on further refinements of those techniques. It seems that any attempt to prove a result of similar nature for words that are not multilinear commutator words would require a different approach.

\section{Preliminary results}\label{sec:prelim}  
 
Throughout, we use the same symbol to denote a group-theoretical property and the class of groups with that property. If $\C$ is a class of 
groups, a virtually-$\C$ group is a group with a normal $\C$-subgroup of finite index. 
The class of virtually-$\C$ groups will be denoted by $\C\F$. 

Let $\C$ be a class of groups closed under taking subgroups, quotients, and such that  in any group the product of finitely many 
normal $\C$-subgroups is again a  $\C$-subgroup. For instance $\C$ is the class of nilpotent, soluble, or finite groups. The next two lemmas are analogues of 
Lemma 2.2 of \cite{DMS-nilpotent} and Lemma 2.6 of \cite{DMS-revised}, respectively. Therefore we omit their proofs.

\begin{lemma} \label{normalclosure} In any group a product of finitely many normal $\CF$-subgroups is again in $\CF$.
\end{lemma}
%\begin{proof} It is sufficient to prove the lemma for a product of two normal $\CF$ subgroups $N_1$ and $N_2$ of a group $G$. By our hypothesis on the class $\C$, there exists a unique maximal normal $\C$-subgroup $R_i$ of $N_i$ for each $i=1,2$. Then $R_1, R_2$ are normal in $G$ and so $R_1R_2$ is a normal $\C$-subgroup of finite index in $N_1N_2$.  \end{proof}

If $A$ is a subset of a group $G$, we write $\langle A \rangle$ for the subgroup generated by $A$. If $B$ is another subset, 
we  denote by $A^B$ the set $\{a^b\mid a\in A\textrm{ and } b\in B\}.$

% Next lemma is the analogue of  Lemma 2.6 in \cite{DMS-revised}. See  \cite{DMS-revised} for the proof, which can be carried over verbatim.
 \begin{lemma}
Let $L$ be a subgroup of a profinite group $G$ such that the normalizer $N_G(L)$ is open. 
\begin{enumerate}
\item If $L$ is finite, then $\langle L^G\rangle$ is finite. 
\item If $L$ is in $\C$ and $H$ is a normal open subgroup of $G$ contained in $N_G(L)$, 
then $\langle (L\cap H)^G\rangle$ is in $\C$. 
\end{enumerate}
\end{lemma}

Throughout this section $w=w(x_1,\dots,x_n)$ is a  multilinear commutator.
Multilinear commutators are words which are obtained by nesting commutators, but using always different variables. More formally, the word $w(x) = x$ in one variable is a multilinear commutator; if $u$ and $v$ are  multilinear commutators involving different variables then the word $w=[u,v]$ is a multilinear commutator, and all multilinear commutators are obtained in this way.
 
An important family of  multilinear commutators is formed by so-called derived words $\delta_k$,  
  on $2^k$ variables,   defined recursively by
$$\delta_0=x_1,\qquad \delta_k=[\delta_{k-1}(x_1,\ldots,x_{2^{k-1}}),\delta_{k-1}(x_{2^{k-1}+1},\ldots,x_{2^k})].$$
 Of course  $\delta_k(G)=G^{(k)}$ is the $k$-th  term of the derived series of $G$. 

We recall the following well-known result (see for example \cite[Lemma 4.1]{S2}). 
\begin{lemma}\label{lem:delta_k} Let $G$ be a group and let $w$ be a multilinear commutator on $n$ variables. Then each $\delta_n$-value is a $w$-value.
\end{lemma}
 
If $A_1,\dots,A_n$ are subsets of a group $G$, we write
$$\X_w(A_1,\dots,A_n)$$ to denote the set of all $w$-values $w(a_1,\dots,a_n)$ with $a_i\in A_i$. Moreover, we write $w(A_1, \dots , A_n)$ for the subgroup $\langle\X_w(A_1,\dots,A_n)\rangle$. Note that if every $A_i$ is a normal subgroup of $G$, then $w(A_1, \dots , A_n)$ is normal in $G$.
  
Let $I$ be a subset of $\{1,\dots,n\}$. % and let $I^\mathsf{c}=\{1, \dots ,n \} \setminus I.$ 
  Suppose that we have a family $A_{i_1}, \dots , A_{i_s}$ of subsets of $G$ with indices running  over $I$ and another family 
  $B_{l_1}, \dots , B_{l_t}$ of subsets with indices  running  over $\{1, \dots ,n \} \setminus I.$ 
 We write 
 $$w_I(A_i ; B_l)$$ 
 for $w(X_1, \dots , X_n)$, where $X_k=A_k$ if $k \in I$, and $X_k=B_k$  otherwise. 
 On the other hand, whenever $a_i\in A_i$ for $i\in I$ and $b_l\in B_l$ for $l\in \{1,\dots,n\}\setminus I$, the symbol 
 $w_I(a_i;b_l)$ stands for the element $w(x_1, \dots , x_n)$, where $x_k=a_k$ if $k \in I$, and $x_k=b_k$ otherwise.

\begin{lemma}\label{zerobis}
Assume that $G$ is  a group and $A_1,\dots,A_n,H$ are normal subgroups of $G$.
Let $a_i\in A_i$ 
% $a_1\in A_1,\dots,a_n\in A_n$
 and  $h_i\in H\cap A_i$ for every $i=1,\dots,n$. 
 Let   $j \in \{1,\dots,n\}$ and set  $I=\{1,\dots,n\}\setminus\{j\}$. 
 Then there exists  an element 
 $$x \in \X_w( a_1(H\cap A_1), \dots, a_n(H\cap A_n))$$
  such that 
% $\tilde h_1\in H\cap A_1,\dots, \tilde h_n\in H\cap A_n$ such that
$$w(a_1h_1,\dots,a_nh_n)= x \cdot  w_I(a_ih_i;h_j).$$         %w(a_1\tilde h_1,\dots,a_n\tilde h_n)w_I(a_ih_i,h_j).$$
\end{lemma}

\begin{proof}
The proof is by induction on the number of variables $n$ appearing in $w$.
 If $n=1$ then $w(a_1h_1)=a_1h_1$ and the statement  is trivially true.
 
So assume that $n\geq2$ and let $w=[w_1,w_2]$ where $w_1,w_1$ are multilinear commutators in $s$ and $n-s$ variables, respectively. Write
$$y=w(a_1h_1,\dots,a_nh_n)=[y_1,y_2]$$ where $y_1=w_1(a_1h_1,\dots,a_sh_s),$  and $y_2=w_2(a_{s+1}h_{s+1},\dots,a_nh_n)$.

Assume also that $j>s$. Then by induction 
$y_2=xh$, where 
$x \in \X_{w_2} (a_{s+1}(H\cap A_{s+1}),\dots,a_n (H\cap A_n))$ and 
$$h=w_2(a_{s+1}h_{s+1},\dots,a_{j-1}h_{j-1},h_j,a_{j+1}h_{j+1},\dots,a_nh_n).$$
%and  $b =w(a_{s+1}\tilde h_{s+1},\dots,a_n\tilde h_n),$ for some $\tilde h_{s+1}\in H\cap A_{s+1},\dots,\tilde h_n\in H\cap A_n$. 
So
$$y=[y_1,y_2]=[y_1,xh]=[y_1, h][y_1, x]^h=[y_1,x]^{h [y_1,h]^{-1}} [y_1,h].$$
Since $\tilde h=h [y_1,h]^{-1} \in H$  and  $a_i \in A_i$, clearly $[a_i,\tilde h]\in H\cap A_i$ and so 
 $$(a_i \tilde h_i)^{\tilde h}= a_i[a_i,\tilde h]h_i^{\tilde h} \in a_i ( H\cap A_i),$$ 
  for every $\tilde h_i \in H\cap A_i$ and every $i$. 
 As $x  =w_2(a_{s+1}\tilde h_{s+1},\dots,a_n\tilde h_n),$ for some $\tilde h_{i}\in H\cap A_{i},$
 %\in \X_{w_2} (a_{s+1}(H\cap A_{s+1}),\dots,a_n (H\cap A_n))$, 
 it follows that 
 \begin{eqnarray*}
 [y_1,x]^{\tilde h} &=& w(a_1h_1,\dots,a_sh_s, a_{s+1} \tilde h_{s+1},\dots,a_n \tilde h_n)^{\tilde h} \\
 &=& w((a_1h_1)^{\tilde h},\dots,(a_sh_s)^{\tilde h}, (a_{s+1} \tilde h_{s+1})^{\tilde h},\dots,(a_n \tilde h_n)^{\tilde h})
\end{eqnarray*} 
 belongs to $ \X_w( a_1(H\cap A_1), \dots, a_n(H\cap A_n)),$
 as desired.

The case $1\le j\le s$ is similar.
By induction $y_1=xh$, where 
$$h=w_1(a_{1}h_{1},\dots,a_{j-1}h_{j-1},h_j,a_{j+1}h_{j+1},\dots,a_sh_s)$$ and 
$x \in \X_{w_1} (a_{1}(H\cap A_{1}),\dots,a_s (H\cap A_s)).$
%w(a_{1}\tilde h_{1},\dots,a_s\tilde h_s),$ for some $\tilde h_{1}\in H\cap A_{s},\dots,\tilde h_s\in H\cap A_s$. 
So
$$%\begin{eqnarray*}
y=[y_1,y_2]=[xh,y_2]=[x,y_2]^h[h,y_2]. $$%\end{eqnarray*}
Note that $h\in H$ and $a_i \in A_i$, therefore, as above, 
$(a_i \tilde h_i)^{ h} \in a_i ( A_i\cap H)$ for every $\tilde h_i\in H\cap A_i$ and every $i$. So $$[x,y_2]^h \in\X_w(a_1(H\cap A_1),\dots, a_n(H\cap A_n))$$ and the result follows.
\end{proof}

\begin{lemma}\label{uno} Let $H,A_1,\dots,A_n$ be normal subgroups of a group $G$. Let $V$ be a subgroup of $G$ and $g\in G$. Assume that for some elements $a_i\in A_i$, the following holds:
$$\X_w( a_1(H\cap A_1),\dots,a_n(H\cap A_n)) 
\subseteq gV.$$
Let $I$ be a proper subset  of $\{1,\dots,n\}$. 
Then 
$$w_I(a_i (H\cap A_i);H\cap A_l )\le V.$$
\end{lemma} 
\begin{proof}
The proof is by  induction on $n-|I|$, 
 so first assume that $I=\{1,\dots,n \}\setminus\{j\}$ for some index $j$. 

 We will write for short $H_i=H\cap A_i$, for every $i=1, \dots ,n$. 
 
Consider $ w(g_1,  \dots , g_n)$, where $g_i \in a_i  H_i   $ for every $i\ne j$ and $g_j  \in H_j$. By Lemma \ref{zerobis} we have $$w(g_1,  \dots ,g_{j-1},a_jg_j,g_{j+1},\dots,g_n)=xw(g_1,\dots,g_n),$$
for some $x\in \X_w( a_1  H_1   ,\dots,a_n  H_n   )\subseteq gV$. As  
$$w(g_1,\dots,a_jg_j,\dots,g_n)\in\X_w(a_1  H_1   ,\dots,a_n  H_n   ),$$ it follows that $w(g_1,\dots,g_n)\in V$. Since $V$ is subgroup, we deduce that $w_I(a_i   H_i   ;H_l )\le V$ and this concludes the case $|I |=n-1$.  

Now assume that $|I|\le n-2$ and let $I^*=I\cup\{j\}$ for some $j\notin I$.
Consider $w(g_1,\dots,g_n)$, where $g_i\in H_i$ for every $i\in I$ and  $g_i \in a_i  H_i   $ for every $i\not\in I$. Then the element $w(g_1,\dots,g_{j-1},a_jg_j,g_{j+1},\dots,g_n)$ belongs to $ w_{I*}(a_i  H_i   ;H_l )$. By Lemma \ref{zerobis} we have $$w(g_1,\dots,g_{j-1},a_jg_j,g_{j+1},\dots,g_n)=xw(g_1,\dots,g_n),$$  for some 
$$x\in \X_w( g_1  H_1   ,\dots ,g_{j-1}  H_{j-1}   ,a_j  H_j   ,g_{j+1}  H_{j+1}   , \dots, g_n  H_n   ).$$ In particular $x \in  w_{I*}(a_i   H_i   ;H_l )$. 
Since, by induction,
  $ w_{I*}(a_i   H_i   ;H_l )
\le V$, it follows that $ w(g_1,\dots,g_n)\in V$, as we wanted. The proof is complete.
\end{proof} 

By applying the previous lemma with $I=\emptyset$ and $A_i=G$ for each $i$,   we obtain the following corollary. 

\begin{corollary}\label{w(H)}
Let $G$ be a group, $H$ and $V$ subgroups of $G$, and $g\in G$. Assume that $H$ is normal and $$\X_w(a_1H,\dots,a_nH) \subseteq gV$$ for some elements
$a_1,\dots,a_n\in G$. Then $w(H)\subseteq V$.
\end{corollary}
The next lemma is Lemma 4.1 in \cite{DMS-revised}.

\begin{lemma}\label{M} 
Let $A_1,\dots,A_n$ and $H$ be normal subgroups of  a group $G$. 
Let $I$ be a subset of $\{1, \dots ,n \}$. Assume that for every proper subset $J$ of $I$ 
\[ w_J (A_i; H\cap A_l)=1.\]
Suppose we are given elements  $g_i \in A_i$ with $i \in I$ and  elements $h_k \in H\cap A_k$ with $k \in \{1, \dots, n\}$. 
 Then we have 
\[w_I(g_ih_i; h_l)=w_I(g_i;h_l).\] 
\end{lemma}

We will now introduce some more notation to handle some particular properties of multilinear commutators. We denote by $\mathbf{I}$ the set of $n$-tuples $(i_1,\dots,i_n)$, where all entries $i_k$ are non-negative integers. We will view $\mathbf{I}$ as a partially ordered set with the partial order given by the rule that $$(i_1,\dots,i_n)\leq(j_1,\dots,j_n)$$ if and only if $i_1\leq j_1,\dots,i_n\leq j_n$.

Given 
 $\mathbf{i}=(i_1,\ldots,i_n) \in \mathbf{I}$, we write
\[
w(\mathbf{i})=w(G^{(i_1)},\ldots,G^{(i_n)})
\]
 for  the subgroup generated by the $w$-values $w(g_1,\dots,g_n)$ with  $g_{j} \in G^{(i_j)}$. 
Further, let 
\[
w(\mathbf{i^+})=\prod w(\mathbf{j} ),
\]
 where the product is taken over all $\mathbf{j} \in \mathbf{I} $ such that $\mathbf{j}>\mathbf{i}$.

\begin{lemma}\label{lem:ab} \cite[Corollary 6]{DMS1}
 Let $w=w(x_1,\ldots,x_n)$ be a multilinear commutator\  and let $\mathbf{i} \in \mathbf{I}$. If $ w(\mathbf{i^+})=1$, then $ w(\mathbf{i})$ 
 is abelian.  
\end{lemma}

The following lemma is Proposition 7 in \cite{DMS1}.
  \begin{lemma}\label{pow_old}
Let $\mathbf{i}=(i_1,\ldots,i_n) \in \mathbf{I}$ and suppose that $w(\mathbf{i^+})=1$. 
If  $a_j\in G^{(i_j)}$ for $j = 1,\dots, n$, and   $b_s\in  G^{(i_s)}$ 
 then 
$$w(a_1 ,\dots, a_{s-1}, b_s a_s , a_{s+1} ,\dots, a_k )$$
$$\quad=w(\tilde a_1 ,\dots, \tilde a_{s-1}, b_s ,\tilde  a_{s+1} ,\dots,\tilde  a_k )w(a_1 ,\dots, a_{s-1}, a_s , a_{s+1} ,\dots, a_k ),$$
where $\tilde a_j$ is a conjugate of $a_j$ and moreover $\tilde a_j=a_j$ if $i_j\le i_s$ .
 \end{lemma}
 
\begin{corollary}\label{pow}
Assume that $w(\mathbf{i^+})=1$ and let $a_j\in G^{(i_j)} $ for $j=1,\dots,n$. Let $l$ be an integer. Then $w(a_1,\dots,a_n)^l=w(b_1,\dots,b_n)$ for some $b_1,\dots,b_n$ with $b_j\in G^{(i_j)}$.
\end{corollary}
\begin{proof}
 Let $i_s$ be maximal among all $i_j$'s, with $j = 1,\dots, k$. 
 Note that by Lemma \ref{pow_old} for every  $a_j\in G^{(i_j)}$, where 
 $j=1,\dots,n$, and every $b_s\in G^{(i_s)}$ we have:
 $$w(a_1 ,\dots, a_{s-1}, b_s a_s , a_{s+1} ,\dots, a_k )$$
$$\quad=w(a_1 ,\dots, a_{s-1}, b_s ,a_{s+1} ,\dots,a_k )w(a_1 ,\dots, a_{s-1}, a_s , a_{s+1} ,\dots, a_k ).$$ It follows that $$w(a_1,\dots,a_{s-1},  a_s,a_{s+1},\dots,a_k )^l=w(a_1,\dots,a_{s-1},a_s^l,a_{s+1},\dots,a_k )$$ for every integer $l$. This proves the result.
\end{proof}
Recall that an element of a group $G$ is called an $FC$-element if it has only finitely many conjugates in $G$. The next result is Lemma 2.7 in \cite{DMS-nilpotent}.  

 \begin{lemma}\label{FC}
Let $G=\langle H, a_1,\dots,a_s\rangle$ be a profinite group, where $H$ is an open abelian normal 
subgroup and  $a_1, \dots, a_s$ are FC-elements. Then $G'$ is finite.
 \end{lemma}

\section{Proof of the main theorem}

Recall that $\C$ is a class of groups closed under taking subgroups, quotients, and such that in any group the product of finitely many normal $\C$-subgroups is again a $\C$-subgroup.
 
Throughout this section we will work under the following hypothesis:

\begin{hypothesis}\label{hyp} Let $w=w(x_1,\dots,x_n)$ be a multilinear commutator and let $G$ be a profinite group in which the set of $w$-values is contained in a union of countably many cosets $\g_iG_i$ of
subgroups $G_i$, where each $G_i\in\C$. 
\end{hypothesis}

\begin{lemma}\label{H} Assume Hypothesis \ref{hyp}. Then $G$ contains an open normal subgroup $H$ such that $w(H)$ is in $\C$.
\end{lemma}
\begin{proof}
For each positive integer $i$ consider the set  \[S_i=\{ (g_1, \dots, g_n)\in G\times \dots \times G \mid  w(g_1, \dots, g_n)\in \g_iG_i \}.\]
Note that the sets $S_i$ are closed in $G\times \dots \times G$ and cover the whole group $G\times \dots \times G$. By the Baire category theorem at least one of these sets has non-empty interior. Hence, there exist an open 
normal subgroup $H$ of  $G$, elements $a_1,\dots a_n \in G$, and an integer $j$ such that $w(a_1H,\dots,a_nH)\subseteq\g_jG_j$.
By  Corollary \ref{w(H)} we have $w(H)\le G_j$, so the result follows.
\end{proof}
 
\begin{lemma}\label{H_a-multi}  Assume Hypothesis \ref{hyp} and let $a\in G$ be a $w$-value. There exists a normal open subgroup $H_a$ in $G$ such that $[H_a,a]$ is in $\C$.
\end{lemma}
\begin{proof} For each positive integer $i$ let 
\[S_i=\{ x \in G  \mid  a^x\in \g_iG_i\}.\]
Note that the sets $S_i$ are closed in $G $ and cover the whole group $G$. By the Baire category theorem at least one of these sets has non-empty interior. Hence, there exist an open normal subgroup $H$ 
of $G$, an element $b\in G$, and an integer $j$  such that $a^{hb}\in \g_jG_j$ for any $h\in H$. Of course we can assume that $\g_j=a^b$, so that   $a^{-b} a^{hb}\in G_j$ for every $h\in H$. Thus $a^{-1}a^h\in G_j^b$ for every $h\in H$. Hence, $[a,H]=[H,a]\le G_j^b$ is in $\C$. 
\end{proof}

Recall that $G^{(i)}$ denotes the $i$-th term of the derived series of a group $G$.
\begin{proposition}\label{delta} Assume Hypothesis \ref{hyp}. Then $G^{(2n)}$ is in $\C\F$.
\end{proposition}

\begin{proof} By Lemma \ref{H} there exists an open normal subgroup $H$ such that $w(H)$ is in $\C$. Lemma \ref{lem:delta_k} implies that $H^{(n)}$ is in $\C$. Let $K=G^{(n)}$ and $L=K\cap H$. Note that $L$ is open in $K$. Choose a  finite set of $\delta_n$-values $a_1,\dots,a_s$ such that
 $K=\langle L, a_1,\dots,a_s\rangle$ and let $H_{a_1},\dots,H_{a_s}$ be normal open subgroups of $G$ such that $[H_{a_j}, a_j]$ is in $\C$ for every $j$ (see Lemma \ref{H_a-multi}). Note that for each $j$ the subgroup $[H_{a_j}, a_j]$ is a normal subgroup of $H_{a_j}$ so $\langle[H_{a_j},a_j]^G \rangle$ is in $\C$. Let $N_1\le K$ be the subgroup generated by $L^{(n)}$ and the subgroups $\langle [H_{a_j},a_j]^G\rangle $ for $j=1,\dots,s$. Note that $N_1$ is in $\C$. The images of $a_1,\dots,a_s$ in the quotient $G/N_1$ are $FC$-elements while the image of $L$ in $G/L'$ is abelian. Therefore by 
  Lemma \ref{FC} the group $KN_1/L'N_1$ 
 has finite 
 derived group. In other words $L'N_1$ has finite index in $K'N_1$. 
 In particular there exist finitely many  $\delta_n$-values $b_1,\dots,b_t$ such that 
   $K'N_1=\langle L', b_1,\dots, b_t, N_1\rangle$.  
   
    As above, there exist   normal open subgroups 
 $H_{b_1},\dots,H_{b_t}$ of $G$
 such that $\langle [H_{b_j},b_j]^G\rangle $ is in $\C$ for every $j$.
  Let $N_2$ be the subgroup generated by $N_1$ 
 and the subgroups $\langle [H_{b_j},b_j]^G\rangle $ for $j=1,\dots,t$. Note that
 $N_2$ is in $\C$. 
  Again, $b_1N_2,\dots,b_tN_2$ are FC-elements in $G/N_2$ and arguing as before we
 obtain that $L^{(2)}N_2$ has finite index in $K^{(2)}N_2$. By iterating this argument we get that $L^{(n)}N_n$ has finite index in $K^{(n)}N_n$
 for some normal $\C$-subgroup $N_n$, so $L^{(n)}(K^{(n)}\cap  N_n)$ has finite index in $K^{(n)}=G^{(2n)}$. 
 As $L^{(n)}\le H^{(n)}$ is in $\C$ it follows that $G^{(2n)}$ is in $\C\F$, as desired.
 \end{proof}
 
Recall the notation introduced in Section \ref{sec:prelim}: whenever $I$ is a subset
 of $\{1, \dots ,n \}$ and $A_{i_1}, \dots , A_{i_s}$ and   $B_{l_1}, \dots , B_{l_t}$  are families 
 of subsets of $G$ with indices running  over $I$ and  $\{1, \dots ,n \} \setminus I$, respectively,  we  write 
  $$w_I(A_i; B_l)$$ 
 for the subgroup $w(X_1, \dots , X_n)$, where $X_k=A_k$ if $k \in I$, and $X_k=B_k$ otherwise. %if $k \notin  I$. 
Moreover, whenever $a_i\in A_i$ for $i\in I$ and $b_l\in B_l$ for $l\in \{1,\dots,n\}\setminus I$, the symbol 
 $w_I(a_i;b_l)$ stands for the element $w(x_1, \dots , x_n)$, where $x_k=a_k$ if $k \in I$, and $x_k=b_k$ otherwise.
 
 Furthermore, given 
 $\mathbf{i}=(i_1,\ldots,i_n) \in \mathbf{I}$, we write
\[w(\mathbf{i})=w(G^{(i_1)},\ldots,G^{(i_n)})\]
 for  the subgroup generated by the $w$-values $w(g_1,\dots,g_n)$ with  $g_{j} \in G^{(i_j)}$ and we set 
$w(\mathbf{i^+})=\prod w(\mathbf{j} )$, 
 where the product is taken over all $\mathbf{j} \in \mathbf{I} $ such that $\mathbf{j}>\mathbf{i}$.

\begin{lemma}\label{step1}
Assume Hypothesis \ref{hyp}.  Let  $A_1,\dots,A_n$ be normal subgroups of $G$  and let $I$  be a  proper 
subset of $\{1, \dots, n\}$. 
Assume that there exist  a normal $\C\F$-subgroup $T$ of $G$ and an open normal subgroup $H$  such that:  
\begin{itemize}
\item[(*)]  %$w(\mathbf{i^+}) \le T$. 
%\item[(**) ]   %$w_J (G^{(i_j)}; H\cap G^{(i_l)}) \le T,$ for every proper subset $J$ of $I$. 
$w_J (A_i; H\cap A_l) \le T$ for every proper subset $J$ of $I$.
\end{itemize}
Then for any given set of elements $\{g_i \}_{i\in I }$, where $g_i \in A_i$,  there exist  an open normal subgroup $U$  of $G$, contained in $H$, and a normal $\C\F$-subgroup $N$ of $G$, containing $T$, such that 
$$w_I(g_i; U\cap A_l) \le N.$$ 
\end{lemma}
\begin{proof}
Consider the sets 
\[S_j=\{(h_1,\dots,h_{n})
\mid h_k\in H\cap A_k \;{\textrm{and}}\; w_I (g_i h_i ; h_l)\in g_jG_j \}.\]
Note that the sets $S_j$ are closed in the group $(H \cap A_1)\times \cdots\times( H \cap A_n)$ and cover 
the whole group. By the Baire  category theorem at least one of these sets has non-empty interior. Hence, there exist
an integer $r$, open subgroups
$V_k$ of $H \cap A_k$, and elements $b_k\in H \cap A_k$ for every 
$k=1,\dots,n$  such that 
$$w_I(g_i b_iv_i; b_lv_l)\in \g_rG_r,$$ 
 for every $v_i \in V_i$.  
 Each subgroup $V_k$ is of the form $V_k=U_k\cap H\cap A_k$ where $U_k$ is an open subgroup
of $G$ and we can assume that $U_k$ is normal in $G$.
Let $U=U_1\cap\dots\cap U_{n}\cap H$. Note that $U$ is an open normal subgroup of $G$ contained
in $H$. 
Thus 
$$w_I(g_i b_iu_i; b_lu_l)\in \g_rG_r,$$
 for every $u_i \in  U \cap A_i$.  
 Now we apply  Lemma \ref{uno} to pass from the cosets $b_l(U \cap A_l)$ to the subgroups  $U \cap A_l$, for every $l \notin I$. 
It follows from Lemma \ref{uno}
  that 
 the subgroup  
$$K=w_I(g_i b_i(U\cap A_i); U\cap A_l)$$
 is contained in  $G_r$ and so it is in $\C\F$. 
Note that $K \le U$. Since $U$ has finite index in $G$ and normalizes $K$, by Lemma \ref{normalclosure}, $\langle K^G\rangle$ is in $\C$. 

Set $N= T \langle K^G\rangle$ and note that $N\in\C$. Using (*) and the fact 
 that $T \le N$ and $b_i(U\cap A_i)\subseteq H\cap A_i$, we can apply Lemma \ref{M} to the group $G/N$. Therefore $$w_I(g_i ; U\cap A_l)N=w_I(g_i b_i (U\cap A_i); U\cap A_l)N.$$ Since $w_I(g_i b_i (U\cap A_i); U\cap A_l) \le N$,   we deduce that $$w_I(g_i  ; U\cap A_l) \le N,$$ as desired. 
 \end{proof}

\begin{lemma}\label{step2} Assume Hypothesis \ref{hyp}. Let  $A_1,\dots,A_n$ be normal subgroups of $G$  and let $I$  be a  proper subset of $\{1, \dots,n\}$. Assume that there exist  a normal $\C\F$-subgroup $T$ of $G$ and an open normal subgroup $H$  such that:  
\begin{itemize}
\item[(*)]  %$w(\mathbf{i^+}) \le T$. 
%\item[(**) ]   %$w_J (G^{(i_j)}; H\cap G^{(i_l)}) \le T,$ for every proper subset $J$ of $I$. 
$w_J (A_i; H\cap A_l) \le T$ for every proper subset $J$ of $I$.
\end{itemize}
Then there exist  an open normal subgroup $U$  of $G$, contained in $H$, and a normal $\C\F$-subgroup $N$ of $G$, containing $T$, such that 
$$w_I( A_i; U\cap A_l) \le N.$$ 
\end{lemma}
\begin{proof}
For each $i \in I$ choose a set $R_i$ of coset representatives of $H\cap A_i$ in $A_i$. Note that all those sets are finite. We apply Lemma \ref{step1} to each choice of elements $\bar g =\{g_i\}_{i \in I},$ with  $g_i \in R_i $: let $U_{\bar g}$ and $N_{\bar g}$ be normal subgroups of $G$ such that $w_I(g_i ; U_{\bar g} \cap A_l) \le N_{\bar g}$. The existence of the subgroups $U_{\bar g}$ and $N_{\bar g}$ is guaranteed by Lemma \ref{step1}. Remark that there are only a finitely many subgroups $U_{\bar g}$ and $N_{\bar g}$. Then  $U =\cap_{\bar g} U_{\bar g}$ is a normal open subgroup of $G$ contained in $H$  and $N= \prod_{\bar g} N_{\bar g}$ is a normal 
$\C\F$-subgroup  containing $T$, such that 
 $$w_I(g_i ; U \cap A_l)\le N$$ for every choice of $g_i\in R_i$. 
Note that, by condition (*) and  Lemma \ref{M}, 
\[ w_I(g_i (H \cap A_i); U \cap A_l) =  w_I(g_i ; U \cap A_l) \le N. \]
Since  $A_i = \cup_{g_i \in R_i} g_i (H \cap A_i)$ for every $i \in I$, we conclude that  
 $$w_I (A_i ; U \cap A_l )= \langle \cup_{\bar g} w_I(g_i (H \cap A_i); U \cap A_l)\rangle  
 \le N,$$
 as desired. 
\end{proof}

\begin{lemma}\label{basic-step}
 Assume Hypothesis \ref{hyp}. 
%Let $\mathbf{i} \in \mathbf{I}$.  
Assume that there exist an $n$-tuple  $\mathbf{i} \in \mathbf{I}$,  a normal $\C\F$-subgroup $T$ of $G$ and an open normal subgroup $H$  such that:  
\begin{itemize}
\item $w(\mathbf{i^+}) \le T$. 
\item  $w (H) \le T.$
\end{itemize}
 Then  
$w( \mathbf{i} )$ is in $\C\F$. 
\end{lemma}
\begin{proof} Let $\mathbf{i}=(i_1,\ldots,i_n)$. 
 We will write for short $$A_j = G^{(i_j)},$$ for every $j= 1,\dots,n$.
It is enough to prove the following statement: 
 for every subset $I$ of $\{1, \dots , n\}$, there exist an open normal subgroup $U_I$ of $G$ contained in $H$ and a normal 
 $\C\F$-subgroup $N_I$ containing $T$ such that 
 $w_I( A_i; U_I\cap A_l) \le N_I.$
 
The proof is by induction on the size $k$ of $I$. If $k=0$, then $I= \emptyset$ and $$ w_\emptyset(A_i ; H \cap A_i)= w(H\cap A_1, \dots , H \cap A_n) \le w(H)\le T.$$

So assume $k>0$. Let $J_1, \dots J_s$ 
 be  all the  proper subsets of $I$. By induction, for each $t=1,\dots,s$  there exist an open normal subgroup $U_{t}$ of $G$ contained in $H$ and a normal  $\C\F$-subgroup $N_{t}$ containing $T$ such that $w_{J_t}( A_i; U_{t}\cap A_l) \le N_{t}$. Let $U=\cap_t U_{t}$ and $N=\langle N_{t}\vert t =1, \dots , s \rangle$. Then $$w_J( A_i; U \cap A_l) \le N$$ for every proper subset $J$ of $I$. 

If $k\ne n$ we can apply Lemma \ref{step2} to $I$. We obtain that there exist an open normal subgroup $U_I$ of $G$  contained in $H$ and a normal 
$\C\F$-subgroup $N_I$ containing $T$ such that $w_I( A_i; U_I\cap A_l) \le N_I$, as desired. 
 
 So we are left with the case when $k=n$, and thus, by definition, 
$w( A_1,\dots, A_n)=w( \mathbf{i} ).$ 
 
For each $i \in I$ choose a set $R_i$ of coset representatives of $ H\cap A_i$ in $A_i$. Note that all those sets are finite. We pass to the quotient $\bar G=G/N$. By Lemma \ref{M} for each choice of elements $\bar g_1,\dots,\bar g_n$ with  $\bar g_i \in \bar R_i $ and for each $\bar h_1,\dots,\bar h_n\in \bar U\cap \bar A_i$, we have 
$$w(\bar g_1\bar h_1,\dots,\bar g_n\bar h_n)=w(\bar g_1,\dots,\bar g_n).$$
So the set $$\X_w(\bar A_1, \dots ,\bar A_n)$$ is finite.

By Lemma \ref{pow} every power of an element in $\X_w(\bar A_1, \dots ,\bar A_n)$ is again in $\X_w(\bar A_1, \dots ,\bar A_n)$. So every  element in $\X_w(\bar A_1, \dots ,\bar A_n)$ has finite order. Therefore  $w(\bar A_1,\dots,\bar A_n)$ is generated by finitely many elements of finite order, and being abelian by Lemma \ref{lem:ab}, it is actually finite.
It follows that $w(A_1,\dots,A_n)$ is in $\C\F$, as desired.
\end{proof}

We are now ready to complete the proof of Theorem \ref{main}. 

\vspace{8pt}
\noindent {\bf Proof of Theorem \ref{main}} Obviously, if $w(G)$ is in $\C\F$ then the set of $w$-values in $G$ is covered by countably many cosets of $\mathcal{C}$-subgroups. Therefore we only need to show that if the set of $w$-values is covered by countably many $\C$-subgroups then  $w(G)$ is in $\C\F$. 

Thus, assume that the set of $w$-values in $G$ is covered by countably many $\C$-subgroups. Proposition \ref{delta} states that $G^{(2n)}$ is in $\C\F$.
 
 Let $H$ be as in Lemma \ref{H}. Then $w(H)$ is in $\C$. 
  Let $T=G^{(2n)}w(H)$. Then $T$ is in $\C\F$  by Lemma \ref{normalclosure}. 
   Since
 $G^{(2n)}\le T$ it follows that $G/T$ is soluble. 
 
 Thus there exist only finitely many 
${\bf i}\in \mathbf{I}$ such that $w({\bf i})T/T\neq1$. 
%The theorem will be proved by
By induction on the number of such $n$-tuples ${\bf i}$, we will prove that every subgroup $w({\bf i})$ is in $\C\F$. 

Choose ${\bf i}=(i_1,\dots,i_n)\in \mathbf{I}$ such that $w({\bf i})T/T\neq1$ while $w({\bf j})T/T=1$ whenever ${\bf i}<{\bf j}$. 
 %It follows from Lemma \ref{lem:ab} that  $w({\bf i}) T/T$ is abelian.
  Now we apply  Lemma \ref{basic-step} 
 and  we obtain that $w({\bf i})$ is in $\C\F$. Let $N=w({\bf i})T$.
 Then induction on the number of ${\bf j}\in \mathbf{I}$ such that $w({\bf j})\not\le N$ leads us to the conclusion that $w(G)$ is in $\C\F$. \qed
\bigskip

\noindent ACKNOWLEDGMENTS. The problem of studying groups in which $w$-values are covered by countably many cosets of subgroups was suggested to us by J. S. Wilson. We thank him for suggesting the problem. The third author was supported  by  FAPDF and FINATEC.

\end{document}